\documentclass[12pt]{amsart}

\usepackage{color,graphicx,amssymb,latexsym,amsfonts,txfonts,amsmath,amsthm}
\usepackage{pdfsync}
\usepackage{amsmath,amscd}
\usepackage[all,cmtip]{xy}

\usepackage{hyperref}
\hypersetup{
    colorlinks=true,       
    linkcolor=blue,          
    citecolor=blue,        
    filecolor=blue,      
    urlcolor=blue           
}

\input epsf
\newcommand{\s}{\smallskip}

\newtheorem{theo}{Theorem}
\newtheorem{coro}{Corollary}

\newtheorem{lemm}{Lemma}

\newtheorem{rema}{Remark}

\begin{document}

\title{On the connectedness of the branch locus of rational maps}
\author{Ruben A. Hidalgo and Sa\'ul Quispe}


\address{Departamento de Matem\'atica y Estad\'{\i}stica, Universidad de La Frontera. Casilla 54-D, 4780000 Temuco, Chile}
\email{ruben.hidalgo@ufrontera.cl}
\email{saul.quispe@ufrontera.cl}
\thanks{Partially supported by Project Fondecyt 1150003}

\begin{abstract}
Milnor proved that the moduli space ${\rm M}_{d}$ of rational maps of degree $d \geq 2$ has a complex orbifold structure of dimension $2(d-1)$. Let us denote by ${\mathcal S}_{d}$ the singular locus of ${\rm M}_{d}$ and by ${\mathcal B}_{d}$ the
branch locus, that is, the equivalence classes of rational maps with non-trivial holomorphic automorphisms. Milnor observed that we may identify ${\rm M}_2$ with ${\mathbb C}^2$ and, within that identification, that ${\mathcal B}_{2}$ is a cubic curve; so ${\mathcal B}_{2}$ is connected and ${\mathcal S}_{2}=\emptyset$. If $d \geq 3$, then ${\mathcal S}_{d}={\mathcal B}_{d}$. We use simple arguments to prove the connectivity of it.
\end{abstract}

\maketitle

\section{Introduction}
The space ${\rm Rat}_{d}$ of complex rational maps of degree $d \geq 2$ can be identified with a Zariski open set of the $(2d+1)$-dimensional complex projective space ${\mathbb P}_{\mathbb C}^{2d+1}$; this is the complement of the algebraic hypersurface defined by the resultant of two polynomials of degree at most $d$. 

The group of M\"obius transformations ${\rm PSL}_{2}({\mathbb C})$ acts on ${\rm Rat}_{d}$ by conjugation:
$\phi, \psi \in {\rm Rat}_{d}$ are said to be equivalent if there is some $T \in {\rm PSL}_{2}({\mathbb C})$ so that $\psi=T \circ \phi \circ T^{-1}$. The ${\rm PSL}_{2}({\mathbb C})$-stabilizer of $\phi \in {\rm Rat}_{d}$, denoted as ${\rm Aut}(\phi)$, is the group of holomorphic automorphisms of $\phi$. As the subgroups of ${\rm PSL}_{2}({\mathbb C})$ keeping invariant a finite set of cardinality at least $3$ must be finite, it follows that ${\rm Aut}(\phi)$ is finite. Levy \cite{Levy} observed that the order of ${\rm Aut}(\phi)$ is bounded above by a constant depending on $d$.

The quotient space ${\rm M}_{d}={\rm Rat}_{d}/{\rm PSL}_{2}({\mathbb C})$ is the moduli space of rational maps of degree $d$. Silverman \cite{Silv2} obtained that ${\rm M}_{d}$ carries the structure of an affine geometric quotient, Milnor \cite{Milnor} proved that it also carries the structure of a complex orbifold of dimension $2(d-1)$
 (Milnor also obtained that ${\rm M}_{2} \cong {\mathbb C}^{2}$) and Levy \cite{Levy} noted that ${\rm M}_{d}$ is a rational variety. Let us denote by ${\mathcal S}_{d} \subset {\rm M}_{d}$ the singular locus of $M_{d}$, that is, the set of points over which ${\rm M}_{d}$ fails to be a topological manifold. The branch locus of ${\rm M}_{d}$ is the set ${\mathcal B}_{d} \subset {\rm M}_{d}$ consisting of those (classes of) rational maps with non-trivial group of holomorphic automorphisms.

As  ${\rm M}_{2} \cong {\mathbb C}^{2}$, clearly ${\mathcal S}_{2}=\emptyset$.
Using this identification, the locus ${\mathcal B}_{2}$ corresponds to the cubic curve \cite{Fujimura}
$$ 2 x^3 + x^2 y - x^2- 4 y^2 - 8 x y + 12 x + 12 y  -36=0,$$
where the cuspid $(-6,12)$ corresponds to the class of a rational map $\phi(z)=1/z^{2}$ with ${\rm Aut}(\phi) \cong D_{3}$ (dihedral group of order $6$) and all other points in the cubic corresponds to those classes of rational maps with the cyclic group $C_{2}$ as full group of holomorphic automorphisms. In this way, ${\mathcal B}_{2}$ is connected.  

 If $d \geq 3$,  then ${\mathcal S}_{d}={\mathcal B}_{d}$ \cite{MSW}.
In this paper we observe the connectivity of ${\mathcal S}_{d}$. 

\s
\noindent
\begin{theo}\label{conexo}
If $d \geq 3$, then the singular locus ${\mathcal S}_{d}={\mathcal B}_{d}$ is connected.
\end{theo}

\s

The proof of Theorem \ref{conexo} will be a consequence of the description of the loci of classes of rationals maps admitting a given cyclic group of holomorphic automorphisms. For it, in Section \ref{Sec:auto}, we recall known results concerning rational maps with non-trivial holomorphic automorphisms (see for instance \cite{DM, MSW}). Then, we observe that the loci in moduli space consisting of classes of rational maps admitting the cyclic group $C_n$, $n \geq 2$, as group of holomorphic automorphisms is connected and we prove that, for $n \geq 3$ prime, such a locus always intersect the locus corresponding to $C_2$.

\s

Theorem \ref{conexo} states that given any two rational maps $\phi, \psi \in {\rm Rat}_{d}$, both with non-trivial group of holomorphic automorphisms, there is some $\rho \in {\rm Rat}_d$ which is equivalen to $\psi$ and there is a continuous family 
$\Theta:[0,1] \to {\rm Rat}_d$ with $\Theta(0)=\phi$, $\Phi(1)=\rho$ and ${\rm Aut}(\Theta(t))$ non-trivial for every $t$. At this point we need to observe that if ${\rm Aut}(\phi) \cong {\rm Aut}(\psi)$, we may not ensure that ${\rm Aut}(\Theta(t))$ stay in the same isomorphic class; this comes from the existence of rigid rational maps \cite{DM} (in the non-cyclic situation).

\s

\s
\noindent
\begin{rema}
In the $80$'s Sullivan provided a dictionary between dynamic of rational maps and the dynamic of Kleinian groups \cite{Sull}. If we restrict to Klenian groups being co-compact Fuchsian groups of a fixed genus $g \ge2$, then we are dealing with closed Riemann surfaces of genus $g$ whose moduli space ${\mathcal M}_g$ has the structure of an orbifold of complex dimension $3(g-1)$. The branch locus in ${\mathcal M}_g$ is the set of isomorphic classes of Riemann surfaces with non-trivial holomorphic automorphisms. In \cite{BCI} it was proved that in general the branch locus is non-connected; a difference with the connectivity of brach locus for rational maps.
\end{rema}

\s

\section{Rational maps with non-trivial group of holomorphic automorphisms}\label{Sec:auto}
It is well known that a non-trivial finite subgroup of ${\rm PSL}_{2}({\mathbb C})$ is either isomorphic to a cyclic group $C_{n}$ or the dihedral group $D_{n}$ or one of the alternating groups ${\mathcal A}_{4}, {\mathcal A}_{5}$ or the symmetric group ${\mathfrak S}_{4}$ (see, for instance, \cite{Beardon}). So, the group of holomorphic automorphisms of a rational map of degree at least two is isomorphic to one of the previous ones. Moreover, for each such finite subgroup there is a rational map admitting it as group of holomorphic automorphisms \cite{DM}. 

Let $G$ be either $C_{n}$ ($n \geq 2$), $D_{n}$ ($n \geq 2$), ${\mathcal A}_{4}$, ${\mathcal A}_{5}$ or ${\mathfrak S}_{4}$. Let us denote by ${\mathcal B}_{d}(G) \subset {\rm M}_d$ the locus of classes of rational maps $\phi$ with ${\rm Aut}(\phi)$ containing a subgroup isomorphic to $G$. We say that $G$ is admissible for $d$ if ${\mathcal B}_{d}(G) \neq \emptyset$. 

If $G$ is either $C_{n}$ or $D_{n}$ or ${\mathcal A}_{4}$, then there may be some elements in ${\mathcal B}_{d}(G)$ with full group of holomorphic automorphisms non-isomorphic to $G$ (i.e., they admit more holomorphic automorphisms than $G$). If $G$ is either isomorphic to ${\mathfrak S}_4$ or ${\mathcal A}_5$, then every element in ${\mathcal B}_{d}(G)$ has $G$ as its full group of holomorphic automorphisms and it may have isolated points \cite{DM}, so it is not connected in general.

Below we recall a description of those values of $d$ for which $G$ is admissible (the results we present here are described in \cite{MSW}) and we compute the dimension of ${\mathcal B}_{d}(G)$. Our main interest will be in the cyclic case, in which case we present the explicit computations, but we recall the general situation as a matter of completeness.

 \s
 \subsection{Admissibility in the cyclic case}
 In the case $G=C_{n}$, $n \geq 2$, the admissibility will depend on $d$. First, let us observe that if a rational map admits $C_{n}$ as a group of holomorphic automorphisms, then we may conjugate it by a suitable M\"obius transformation so that we may assume  $C_{n}$ to be generated by the rotation $T(z)=\omega_{n} z$, where $\omega_{n}=e^{2 \pi i/n}$. 
 
\s
\noindent
\begin{theo}\label{teociclico}
Let $d,n \geq 2$ be integers. The group $C_{n}$ is admissible for $d$ if and only if $d$ is congruent to either $-1,0,1$ modulo $n$. Moreover, for such values, 
every rational map of degree $d$ admitting $C_{n}$ as a group of holomorphic automorphisms is equivalent to one of the form $\phi(z)=z\psi(z^{n})$, where $$\psi(z)=\frac{\sum_{k=0}^{r} a_{k}z^{k}}{\sum_{k=0}^{r} b_{k}z^{k}} \in {\rm Rat}_{r},$$
satisfies that
\begin{itemize}
\item[(a)] $a_{r}b_{0}  \neq 0$,  if $d=nr+1$;

\item[(b)] $a_{r}\neq 0$ and $b_{0}=0$, if $d=nr$;

\item[(c)] $a_{r}=b_{0}=0$ and $b_{r} \neq 0$, if $d=nr-1$.
\end{itemize}

In the bove case, $C_{n}$ is generated by the rotation $T(z)=\omega_{n} z$.
\end{theo}
\begin{proof}
Let $\phi$ be a rational map admitting a holomorphic automorphism of order $n$. By conjugating it by a suitable M\"obius transformation, we may assume that such automorphism is the rotation $T(z)=\omega_{n} z$.

(1) Let us write $\phi(z)=z \rho(z)$.
The equality $T \circ \phi \circ T^{-1}=\phi$ is equivalent to $\rho(\omega_{n} z)=\rho(z)$. Let 
$$\rho(z)=\frac{U(z)}{V(z)}=\frac{\sum_{k=0}^{l} \alpha_{k}z^{k}}{\sum_{k=0}^{l} \beta_{k}z^{k}},$$
where either $\alpha_{l} \neq 0$ or $\beta_{l} \neq 0$ and $(U,V)=1$.

The equality $\rho(\omega_{n} z)=\rho(z)$ is equivalent to the existence of some $\lambda \neq 0$ so that
$$\omega_{n}^{k}\alpha_{k}=\lambda \alpha_{k}, \quad \omega_{n}^{k}\beta_{k}=\lambda \beta_{k}.$$

By taking $k=l$, we obtain that $\lambda=\omega_{n}^{l}$. So the above is equivalent to have, for $k<l$, 
$$\omega_{n}^{l-k} \alpha_{k}=\alpha_{k}, \quad \omega_{n}^{l-k} \beta_{k}=\beta_{k}.$$

So, if $\alpha_{k} \neq 0$ or $\beta_{k} \neq 0$, then $l-k \equiv 0 \mod(n)$. As $(U,V)=1$, either $\alpha_{0} \neq 0$ or $\beta_{0} \neq 0$; so $l \equiv 0 \mod(n)$. In this way, if $\alpha_{k} \neq 0$ or $\beta_{k} \neq 0$, then $k \equiv 0 \mod(n)$. In this way, $\rho(z)=\psi(z^{n})$ for a suitable rational map $\psi(z)$.

\s
\noindent
(2) It follows from (1) that $\phi(z)=z\psi(z^{n})$, for $\psi \in {\rm Rat}_{r}$ and suitable $r$. We next provide relations between $d$ and $r$. Let us write
$$\psi(z)=\frac{P(z)}{Q(z)}=\frac{\sum_{k=0}^{r} a_{k}z^{k}}{\sum_{k=0}^{r} b_{k}z^{k}},$$
where $(P,Q)=1$ and either $a_{r} \neq 0$ or $b_{r} \neq 0$. In this way,
$$\phi(z)=\frac{zP(z^{n})}{Q(z^{n})}=\frac{z \sum_{k=0}^{r} a_{k}z^{kn}}{\sum_{k=0}^{r} b_{k}z^{kn}}.$$

Let us first assume that $Q(0) \neq 0$, equivalently, $\psi(0) \neq \infty$. Then $\phi(0)=0$ and the polinomials $zP(z^{n})$ and $Q(z^{n})$ are relatively prime. If ${\rm deg}(P) \geq {\rm deg}(Q)$, then $r={\rm deg}(P)$, $\psi(\infty) \neq 0$, $\phi(\infty)=\infty$ and ${\rm deg}(\phi)=1+nr$. If ${\rm deg}(P) < {\rm deg}(Q)$, then $r={\rm deg}(Q)$, $\psi(\infty)= 0$, $\phi(\infty)=0$ and ${\rm deg}(\phi)=nr$.

Let us now assume that $Q(0)=0$, equivalently, $\psi(0)= \infty$. Let us write $Q(u)=u^{l}\widehat{Q}(u)$, where $l \geq 1$ and $\widehat{Q}(0)\neq 0$; so ${\rm deg}(Q)=l+{\rm deg}(\widehat{Q})$. In this case,
$$\phi(z)=\frac{P(z^{n})}{z^{ln-1}\widehat{Q}(z^{n})}$$
and the polinomials $P(z^{n})$ (of degree $n{\rm deg}(P)$) and $z^{ln-1}\widehat{Q}(z^{n})$ (of degree $n{\rm deg}(Q)-1$) are relatively prime. If ${\rm deg}(P) \geq {\rm deg}(Q)$, then $r={\rm deg}(P)$, $\psi(\infty) \neq 0$, $\phi(\infty)=\infty$ and ${\rm deg}(\phi)=nr$. If ${\rm deg}(P) < {\rm deg}(Q)$, then $r={\rm deg}(Q)$, $\psi(\infty)= 0$, $\phi(\infty)=0$ and ${\rm deg}(\phi)=nr-1$.

Summarizing all the above, we have the following situations:
\begin{itemize}
\item[(i)] If $\phi(0)=0$ and $\phi(\infty)=\infty$, then  $\psi(0)\neq \infty$ and $\psi(\infty)\neq 0$; in particular, $d=nr+1$. This case corresponds to have $a_{r}b_{0}  \neq 0$.

\item[(ii)] If $\phi(0)=\infty=\phi(\infty)$, then $\psi(0)=\infty$ and $\psi(\infty) \neq 0$; in which case $d=nr$. This case corresponds to have $a_{r}\neq 0$ and $b_{0}=0$.

\item[(iii)] If $\phi(0)=0=\phi(\infty)$, then $\psi(0)\neq \infty$ and $\psi(\infty)= 0$; in particular, $d=nr$. This case corresponds to have $a_{r}=0$ and $b_{0} \neq 0$. But in this case, we may conjugate $\phi$ by $A(z)=1/z$ (which normalizes $\langle T \rangle$) in order to be in case (ii) above.

\item[(iv)] If $\phi(0)=\infty$ and $\phi(\infty)=0$, then $\psi(0)=\infty$ and $\psi(\infty)=0$; in particular, $d=nr-1$. This case corresponds to have $a_{r}=b_{0}=0$ (in which case $b_{r} \neq 0$ as $\psi$ has degree $r$).
\end{itemize}

\end{proof}

\s
\noindent
\begin{coro}
$C_{2}$ is admissible for every $d \geq 2$.
\end{coro}

\s

The explicit description provided in Theorem \ref{teociclico} permits to obtain the connectivity of ${\mathcal B}_{d}(C_{n})$ and its dimension.

\s
\noindent
\begin{coro}\label{coro0}
If $n \geq 2$ and $C_{n}$ is admissible for $d$, then ${\mathcal B}_{d}(C_{n})$ is connected and  
$${\rm dim}_{\mathbb C}({\mathcal B}_{d}(C_{n}))=\left\{ \begin{array}{ll}
2(d-1)/n, & d \equiv 1 \mod n\\
(2d-n)/n, & d \equiv 0 \mod n\\
2(d+1-n)/n, & d \equiv -1 \mod n
\end{array}
\right.
$$
\end{coro} 
\begin{proof}
(1) By Theorem \ref{teociclico}, the rational maps in ${\rm Rat}_{d}$ admitting a holomorphic automorphism of order $n \geq 2$ are conjugated those of the form $\phi(z)=z\psi(z^{n}) \in {\rm Rat}_{d}$ for $\psi \in {\rm Rat}_{r}$ as described in the same theorem. 

Let us denote by ${\rm Rat}_{d}(n,r)$ the subset of ${\rm Rat}_{d}$ formed by all those rational maps of the $\phi(z)=z\psi(z^{n})$, where $\psi$ satisfies the conditions in Theorem \ref{teociclico}.

If $d=nr+1$, then we may identify ${\rm Rat}_{d}(n,r)$ with an open Zariski subset of ${\rm Rat}_{r}$; 
if $d=nr$, then it is identified with an open Zariski subset of a linear hypersurface of ${\rm Rat}_{r}$; and if $d=nr-1$, then it is identified with an open Zariski subspace of  a linear subspace of codimension two of ${\rm Rat}_{r}$. In each case, we have that ${\rm Rat}_{d}(n,r)$ is connected. As the projection of ${\rm Rat}_{d}(n,r)$ to ${\rm M}_{d}$ is exactly ${\mathcal B}_{d}(C_{n})$, we obtain its connectivity.

(2) The dimension counting. We may see that, if $d=nr+1$, then $\psi$ depends on $2r+1$ complex parameters; if $d=nr$, then $\psi$ depends on $2r$ complex parameters; and if $d=nr-1$, then $\psi$ depends on $2r-1$ complex parameters.
The normalizer in ${\rm PSL}_{2}({\mathbb C})$ of $\langle T \rangle$ is the $1$-complex dimensional group $N_{n}=\langle A_{\lambda}(z)=\lambda z, B(z)=1/z: \lambda \in {\mathbb C}-\{0\}\rangle$. If $U \in N_{n}$, then $U \circ \phi \circ U^{-1}$ will also have $T$ as a holomorphic automorphism. In fact, 
$$A_{\lambda} \circ \phi \circ A_{\lambda}^{-1} (z)=z\psi(z^{n}/\lambda^{n}),$$
$$B \circ \phi \circ B (z)=z/\psi(1/z^{n}).$$

In this way, there is an action of $N_{n}$ over ${\rm Rat}_{r}$ so that the orbit of $\psi(u)$ is given by the rational maps
$\psi(u/t)$, where $t \in {\mathbb C}-\{0\}$, and $1/\psi(1/u)$. In this way, we obtain the desired dimensions.
\end{proof}

\s
 \subsection{Admissibility in the dihedral case}
 Let us now assume $\phi \in {\rm Rat}_{d}$ admits the dihedral group $D_{n}$, $n \geq 2$, as a group of holomorphic automorphisms. Up to conjugation, we may assume that $D_{n}$ is generated by $T(z)=\omega_{n}z$ and $A(z)=1/z$. By Theorem \ref{teociclico}, 
we may assume that $\phi(z)=z\psi(z^{n})$, where  
$$\psi(z)=\frac{\sum_{k=0}^{r} a_{k}z^{k}}{\sum_{k=0}^{r} b_{k}z^{k}} \in {\rm Rat}_{r},$$ 
where either 
\begin{itemize}
\item[(a)] $a_{r}b_{0}  \neq 0$,  if $d=nr+1$;

\item[(b)] $a_{r}\neq 0$ and $b_{0}=0$, if $d=nr$;

\item[(c)] $a_{r}=b_{0}=0$ and $b_{r} \neq 0$, if $d=nr-1$;
\end{itemize} 
with the extra condition that $\psi(z)=1/\psi(1/z)$. 
 This last condition is equivalent to the existence of some $\lambda \neq 0$ so that
 $$\lambda a_{k}=b_{r-k}, \quad \lambda b_{k}=a_{r-k}, \quad k=0,1,...,r.$$

The above is equivalent to have $\lambda \in \{\pm 1\}$ and $b_{k}=\lambda a_{r-k}$, for $k=0,1,...,r$. In particular, this asserts that $a_{r}=0$ if and only if $b_{0}=0$ (so case (b) above does not hold). Also, as the normalizer of the dihedral group $D_n=\langle T(z)=\omega_n z, A(z)=1/z\rangle$ is a finite group, the dimension of ${\mathcal B}_{d}(D_{n})$ is the same as half the projective dimension of those rational maps $\psi$ satisfying (a) or (c). So, we may conclude the following result.

\s
\noindent
\begin{theo}\label{teodihedral}
Let $d,n \geq 2$ be integers. The dihedral group $D_{n}$ is admissible for $d$ if and only if $d$ is congruent to either $\pm 1$ modulo $n$. Moreover, for such values, 
every rational map of degree $d$ admitting $D_{n}$ as a group of holomorphic automorphisms is equivalent to one of the form $\phi(z)=z\psi(z^{n})$, where $$\psi(z)=\pm \frac{\sum_{k=0}^{r} a_{k}z^{k}}{\sum_{k=0}^{r} a_{r-k}z^{k}} \in {\rm Rat}_{r},$$
satisfies that
\begin{itemize}
\item[(i)] $a_{r} \neq 0$,  if $d=nr+1$;

\item[(ii)] $a_{r}=0$ and $a_{0} \neq 0$, if $d=nr-1$.
\end{itemize}

In the above case, $D_{n}$ is generated by the rotation $T(z)=\omega_{n} z$ and the involution $A(z)=1/z$.

If $n \geq 2$ and $D_{n}$ is admissible for $d$, then  
$${\rm dim}_{\mathbb C}({\mathcal B}_{d}(D_{n}))=\left\{ \begin{array}{ll}
(d-1)/n, & d \equiv 1 \mod n\\
(d+1-n)/n, & d \equiv -1 \mod n
\end{array}
\right.
$$
\end{theo}

\s
\noindent
\begin{rema}\label{obs1}
(a) If we are in case (i) and ``+" sign for $\psi$, then $\phi$ fixes both fixed points of $T$ and both fixed points of $A$. But, if we are in case (i) and ``-" sign for $\psi$, then $\phi$ fixes both fixed points of $T$ and permutes  both fixed points of $A$.

(b) If we are in case (ii) and ``+" sign for $\psi$, then $\phi$ permutes both fixed points of $T$ and fixes both fixed points of $A$. But, if we are in case (ii) and ``-" sign for $\psi$, then $\phi$ permutes both fixed points of $T$ and also both fixed points of $A$.

(c) If $n \geq 3$, then cases (i) and (ii) cannot happen simultaneously. Also, in either case, we obtain that ${\mathcal B}_{d}(D_{n})$ has two connected components (they correspond to the choices of the sign ``+" or ``-").
\end{rema}

\s
\subsection{Admissibility of the platonic cases}
Let us now assume that $\phi \in {\rm Rat}_{d}$ admits as group of holomorphic automorphisms either ${\mathcal A}_{4}$, ${\mathcal A}_{5}$ or ${\mathfrak S}_{4}$. 
We may assume, up to conjugation, that (see, for instance, \cite{Beardon})

\begin{enumerate}
\item $\langle T_{3}, B: T_{3}^{3}=B^{2}=(T_{3} \circ A)^{3}=I\rangle  \cong {\mathcal A}_{4}$;
\item $\langle T_{4}, C: T_{4}^{4}=C^{2}=(T_{4} \circ C)^{3}=I\rangle  \cong {\mathfrak S}_{4}$.
\item $\langle T_{5}, D: T_{5}^{5}=D^{2}=(T_{5} \circ D)^{3}=I\rangle  \cong {\mathcal A}_{5}$;
\end{enumerate}
\s
where
$$T_{n}(z)=\omega_{n} z, \quad \omega_{n}=e^{2 \pi i/n},$$
$$A(z)=1/z,$$
$$B(z)=\frac{(\sqrt{3}-1)\left(z+(\sqrt{3}-1)\right)}{2z-(\sqrt{3}-1)},$$
$$C(z)=\frac{(\sqrt{2}+1)\left(-z+(\sqrt{2}+1)\right)}{z+(\sqrt{2}+1)},$$
$$D(z)=\frac{\left(1+\sqrt{2-\omega_{5}-\omega_{5}^{4}}\;\right)\left(-z+\left(1+\sqrt{2-\omega_{5}-\omega_{5}^{4}}\;\right)\right)}{(1-\omega_{5}-\omega_{5}^{4})z+\left(1+\sqrt{2-\omega_{5}-\omega_{5}^{4}}\;\right)}.$$

Working in a similar fashion as done for the dihedral situation, one may obtains the following.
 
 \s
 \noindent
 \begin{theo}[\cite{MSW}] 
 Let $d \geq 2$.
 \begin{enumerate}
 \item ${\mathcal A}_{4}$ is admissible for $d$ if and only if $d$ is odd.
 \item ${\mathcal A}_{5}$ is admissible for $d$ if and only if $d$ is congruent modulo $30$ to either $1$, $11$, $19$, $21$.
 \item ${\mathfrak S}_{4}$ is admissible for $d$ if and only if $d$ is coprime to $6$.
 \end{enumerate}
 \end{theo}

\section{Proof of Theorem \ref{conexo}}

It is clear that ${\mathcal B}_{d}$ is equal to the union of all ${\mathcal B}_{d}(G)$, where $G$ runs over the admissible finite groups for $d$.

If $G$ is admissible for $d$ and $p$ is a prime integer dividing the order of $G$ (so that the cyclic group $C_{p}$ is a subgroup of $G$), then $C_{p}$ is admissible for $d$ and ${\mathcal B}_{d}(G) \subset {\mathcal B}_{d}(C_{p})$. In this way, ${\mathcal B}_{d}$ is equal to the union of all ${\mathcal B}_{d}(C_{p})$, where $p$ runs over all integer primes with $C_{p}$ admissible for $d$. Corollary \ref{coro0} asserts that each ${\mathcal B}_{d}(C_{p})$ is connected. Now, the connectivity of ${\mathcal B}_{d}$ will be consequence of Lemma \ref{lemita} below.

\s
\noindent
\begin{lemm}\label{lemita}
If $p \geq 3$ is a prime and $C_{p}$ is admissible for $d$, then ${\mathcal B}_{d}(C_{p}) \cap {\mathcal B}_{d}(C_{2}) \neq \emptyset$.
\end{lemm}
\begin{proof}
We only need to check the existence of a rational map $\phi \in {\rm Rat}_{d}$ admitting a holomorphic automorphism of order $p$ and also a holomorphic automorphism of order $2$.

First, let us consider those rational maps of the form $\phi(z)=z\psi(z^{p})$, where (by Theorem \ref{teociclico}) we may assume to be of the form
$$\psi(z)=\frac{\sum_{k=0}^{r} a_{k}z^{k}}{\sum_{k=0}^{r} b_{k}z^{k}} \in {\rm Rat}_{r},$$
with
\begin{itemize}
\item[(a)] $a_{r}b_{0}  \neq 0$,  if $d=pr+1$;

\item[(b)] $a_{r}\neq 0$ and $b_{0}=0$, if $d=pr$;

\item[(c)] $a_{r}=b_{0}=0$, if $d=pr-1$.
\end{itemize}

Assume we are in either case (a) or (c). By considering $b_{k}=a_{r-k}$, for every $k=0,1,...,r$, 
we see that $\psi$ satisfies the relation $\psi(1/z)=1/\psi(z)$; so $\phi$ also admits 
the holomorphic automorphism $A(z)=1/z$.  The automorphisms $T(z)=\omega_{p} z$ and $A$ generate a dihedral group of order $2p$. 

In case (b), we can consider $\psi$ so that $\psi(-z)=\psi(z)$, which is posible to find if we assume that $(-1)^{k}a_{k}=(-1)^{r}a_{k}$ and $(-1)^{k}b_{k}=(-1)^{r}b_{k}$ (which means that $a_{k}=b_{k}=0$ if $k$ and $r$ have different parity). In this case $T$ and $V(z)=-z$ are holomorphic automorphisms of $\phi$, generating the cyclic group of order $2p$.

\end{proof}

\s


\end{document}